\documentclass[11pt,oneside]{amsart}
\usepackage{graphics}
\usepackage[dvips]{graphicx}
\usepackage{cancel}
\usepackage{amsmath,amsfonts}
\usepackage[update,prepend]{epstopdf}
\usepackage{times}
\usepackage{float}
\usepackage[ansinew]{inputenc}
\usepackage{amssymb}
\usepackage{amsmath}
\usepackage{subfigure}
\usepackage{amsfonts}
\usepackage{latexsym}
\usepackage{epsfig}
\usepackage{geometry}
\usepackage[update,prepend]{epstopdf}
\usepackage{adjustbox}
\usepackage{makecell}
\usepackage[colorlinks = true,
linkcolor = black, citecolor=black]{hyperref}
\newcommand{\done}{\item[\checkmark]}
\newcommand{\crossed}{\item[$\times$]}
\newcommand{\fonte}[1]{\Large #1}

\author{ M. Astudillo}
\address{Department of Mathematics, Federal University of São Paulo, São José dos Campos, 12231-280, SP, Brazil.}
\email{mastudillo86@gmail.com}

\author{ M. M. Cavalcanti}
\address{ Department of Mathematics, State University of
Maring\'a, 87020-900, Maring\'a, PR, Brazil.}
\email{mmcavalcanti@uem.br}
\thanks{Research of Marcelo M. Cavalcanti is partially supported by the CNPq Grant 300631/2003-0}

\author{ V. N. Domingos Cavalcanti}
\address{ Department of Mathematics, State University of
Maring\'a, 87020-900, Maring\'a, PR, Brazil.}
\email{vndcavalcanti@uem.br}
\thanks{Research of Val\'eria N. Domingos Cavalcanti is partially supported by the CNPq Grant 304895/2003-2}

\author{V. H. Gonzalez Martinez}
\address{Department of Mathematics, State University of
	Maring\'a, 87020-900, Maring\'a, PR, Brazil.}
\email{victor.hugo.gonzalez.martinez@gmail.com}
\thanks{Research of Victor Hugo Gonzalez Martinez was supported by grant CAPES}

\title[Boundary Control for a Generalized Wave Equation]{Boundary Control for a Generalized Wave Equation - Revisiting Russell's Method of Control}
\begin{document}

\renewcommand{\theequation}{\thesection.\arabic{equation}}
\newtheorem{theorem}{Theorem}
\newtheorem{proposition}{Proposition}[section]
\newtheorem{lemma}{Lemma}[section]
\newtheorem{corollary}{Corollary}[section]
\newtheorem{definition}{Definition}[section]
\newtheorem{remark}{Remark}
\newtheorem{assumption}{Assumption} [section]
\newtheorem{example}{Example}[section]
\newtheorem*{acknowledgement}{Acknowledgements}
\newtheorem*{notation}{Notation}

\newenvironment{dem}{\smallskip \noindent{\bf Proof}: }
{\hfill \rule{0.25cm}{0.25cm}}

\def\supp{\operatorname{{supp}}}

\begin{abstract}
In this work we study the exact boundary controllability of a generalized wave equation in a nonsmooth domain with a nontrapping obstacle. In the more general case, this work contemplates the boundary control of a transmission problem admitting several zones of transmission. The result is obtained using the technique developed by David Russell, taking advantage of the local energy decay for the problem, obtained through the Scattering Theory as used by Vodev, combined with a powerful trace Theorem due to Tataru.
\end{abstract}

\maketitle




%
\tableofcontents

\section{Introduction}

\setcounter{equation}{0}

\subsection{Description of the problem and main result}

Let $\Omega \subset \mathbb{R}^n$, $n \geq 2$ be a connected complement of a compact obstacle, $K$, with smooth boundary. Let also $\Omega_{N_0} \subset \Omega_{N_0-1}\cdots \Omega_{1}\subset \Omega_{0}:=\Omega$ be a finite number of open connected domains with smooth boundaries and bounded complements such that $\mathcal{O}_{k}=\Omega_{k-1} \setminus\Omega_{k}$, $k=1,\dots,N_0$, are bounded connected domains.

 Define the Hilbert space $H=\overset{N_0}{\underset{k=1}{\bigoplus}} L^2(\mathcal{O}_k;c_k(x)dx)\oplus L^2(\Omega_{N_0})$, which will be denoted by $L^2(\Omega)$. Let $P_k$, $k=1,\dots, N_0$, be the differential operator defined in $\mathcal{O}_k$, respectively, of the form
\begin{gather}\label{eq1.1.1}
	P_k=-c_k(x)^{-1}\sum_{i,j=1}^{n}\partial_{x_i}(g_{ij}^{(k)}(x)\partial_{x_j}),
\end{gather}
with smooth coefficients. Let $P$ be a self-fadjoint, positive operator on $H$ with absolutely continuous spectrum only, such that
\begin{gather}\label{eq2.1.1}
P|_{\mathcal{O}_{k}}=P_k, \ P|_{\Omega_{N_0}}=-\Delta=-\sum_{j=1}^{n}\partial_{x_j}^{2}.
\end{gather}
 We also suppose that $P$ is elliptic, i.e., the operator
\begin{gather*}
(P+1)^{-m}:H\longrightarrow \overset{N_0}{\underset{k=1}{\bigoplus}} H^{2m}(\mathcal{O}_{k})\oplus H^{2m}(\Omega_{N_0})
\end{gather*}
is bounded for every $m \geq 0$.

Set $R(\lambda)=(P-\lambda^2)^{-1}:H\longrightarrow H$ for $\Im m  \lambda<0$, and let $\chi \in C_{0}^{\infty}(\mathbb{R}^n)$, $\chi=1$ on $B=\{x \in \mathbb{R}^n: |x|\leq \rho_0 \}$, $\rho_0 \gg 1$. Then $R_\chi(\lambda)=\chi R(\lambda)\chi: H \longrightarrow H$ extends to a meromorphic function on $\mathbb{C}$ if $n$ is odd, and on the Riemann surface, $\Lambda$, of $\log \lambda$, if $n$ is even (e.g. see \cite{19}). Suppose that
\begin{gather}\label{eq3.1.1}
\|\lambda R_\chi(-i\lambda)\|<\infty, \quad \lambda \rightarrow 0, \quad \lambda>0,
\end{gather}
where $\|\cdot\|$ denotes the norm in $\mathcal{L}(H,H)$.

Denote by $\widetilde{H}_{k}$, $k=1,\dots,N_0$, the closure of $C_{0}^{\infty}(\mathcal{O}_{k})$ with respect to the norm
\begin{equation*}
\left(\int_{\mathcal{O}_{k}}\sum_{i,j=1}^{n}g_{ij}^{(k)}(x)\partial_{x_i}f\partial_{x_j}\overline{f}dx \right)^{\frac{1}{2}}
\end{equation*}
and by $\widetilde{H}_{N_0+1}$ the closure of $C_{0}^{\infty}(\Omega_{N_0})$ with respect to the norm \begin{equation*}
\left(\int_{\Omega_{N_0}}|\nabla_xf|^2dx\right)^{\frac{1}{2}}.
\end{equation*}
Set $\widetilde{H}(\Omega)=\overset{N_0+1}{\underset{k=1}{\bigoplus}}\widetilde{H}_{k}$, and $\mathcal{H}(\Omega)=\widetilde{H}(\Omega)\oplus H$.
Consider the operator
\begin{equation*}
G=-i\left(
\begin{array}{cc}
0& Id\\
-P&  0\\
\end{array}\right),
\end{equation*}
on the Hilbert space $\mathcal{H}$ with domain of definition
\begin{equation*}
D(G)=\{(u_1,u_2)\in \mathcal{H}:u_1 \in D(P),Pu_1\in H,u_2 \in \widetilde{H}\}.
\end{equation*}
It is easy to see that the operator $G$ is self-adjoint.

Denote by $u(t)$ the solution, obtained with the Stone's Theorem, of the Cauchy problem
\begin{equation}
\begin{cases}\label{eq4.1.1}
(\partial_{t}^{2}+ P)u=0 &~\hbox{ in }~ \Omega \times \mathbb{R}, \\
u = 0 &~\hbox{ on }~\partial \Omega \times \mathbb{R},\\
u(x,0)=f_1(x), ~\partial_t u(x,0)=f_2(x) &~\hbox{ on }~ \Omega.
\end{cases}
\end{equation}

Let $a>\rho_0$ and set $B_a=\{x \in \mathbb{R}^n: |x|\leq a\}$. Given any $m \geq 0$, set
\begin{eqnarray}\label{eq5.1.1}
	p_m(t) &=& \sup \left\{
	\begin{array}{c}
	\displaystyle \frac{\|\nabla_xu\|_{L^2(B_a\cap \Omega)}+\|\partial_t u\|_{L^2(B_a \cap \Omega)}}{\|\nabla_xf_1\|_{H^{m}(B_a \cap \Omega)}+\|f_2\|_{H^m(B_a \cap \Omega)}}, \\
		(0,0)\neq (f_1,f_2) \in  C^\infty(\Omega) \times C^\infty(\Omega), \supp f_j \subset B_a%
	\end{array}%
	\right\}
\end{eqnarray}%

The main result of Vodev \cite{1}, the source of inspiration of the present article, is the following:

\begin{theorem}\label{theo1.1.1}
	The following three statements are equivalent:
	\begin{enumerate}
		\item [i)] $\underset{t \rightarrow +\infty}{\lim}p_0(t)=0.$ 		
		\item [ii)] There exist constants $C,C_1>0$ so that
		\begin{equation*}
		\|\lambda R_\chi(\lambda)\|\leq C, \quad \lambda \in \mathbb{R}, \ |\lambda|> C_1.
		\end{equation*}
		
		\item [iii)] There exist constants $C, \gamma>0$ so that
		\begin{equation*}
		p_0(t)=\begin{cases}
		Ce^{-\gamma t} &~\hbox{ if }~ n \mbox{ \ is odd}, \\
		Ct^{-n} &~\hbox{ if }~ n \mbox{ \ is even}.
		\end{cases}
		\end{equation*}
	\end{enumerate}
\end{theorem}
\begin{proof}
See \cite{1}.
\end{proof}

A natural question that arises is the following: {\it What condition implies the veracity of item ii) above mentioned?} This is a question studied by several authors, as for example \cite{35}, \cite{33}, \cite{34}, \cite{32} and \cite{36}. Finally, using the concept of generalized bicharacteristics, introduced by R. B. Melrose an J. Sjöstrand in \cite{37} and \cite{38}, it was proved in \cite{33} and \cite{39}  that the condition that every generalized geodesic leaves any compact in a finite time is sufficient for ii) to be fulfilled, that is, the metric associated with the equation \eqref{eq4.1.1} must be {\bf non-trapping}. For this reason, we assume this condition, that is:
\begin{assumption}\label{ass1}
Every generalized geodesic leaves any compact in a finite time,
\end{assumption}
\noindent which will imply the local energy decay given by the third item of Theorem \ref{theo1.1.1}.

\begin{remark}\label{Omegastar}
Throughout this article $\Omega^\ast$ will denote a bounded domain of $\mathbb{R}^n$, with  piecewise smooth boundary $\partial \Omega^\ast$ and no cuspidal points, such that $ \Omega^\ast \subset B_a \cap \Omega$. We also assume that $\Omega^\ast$ lies on one side of its boundary. Under these assumptions it follows that the unit vector $\nu$ normal to the boundary, pointing outside, is defined almost everywhere on $\Omega^\ast$. Also, $\Omega^\ast$ is a Lipschitz bounded domain. We denote a $\delta$-neighborhood of $\Omega^\ast$ by $\Omega^{\ast}_{\delta}=\{y \in \mathbb{R}^n: |y-x|<\delta ~\hbox{ for all }~ x \in \Omega^\ast\}$.
\end{remark}
 Define $\widetilde{H}(\Omega^{\ast})= \overset{N_0}{\underset{k=1}{\bigoplus}}\widetilde{H}_{k}(\mathcal{O}_{k}\cap \Omega^{\ast})\oplus \widetilde{H}_{N_0+1}(\Omega_{N_0}\cap \Omega^\ast)$, $L^2(\Omega^\ast)=\overset{N_0}{\underset{k=1}{\bigoplus}}L^2(\mathcal{O}_{k}\cap \Omega^\ast,c_k(x)dx)\oplus L^2(\Omega_{N_0}\cap\Omega^\ast)$ and assume, as has been considered in the pioneer work of Russel \cite{3} (see definition 1.2 in \cite{3}) that,
\begin{assumption}\label{ass2}
There exist a bounded linear operator $P^\ast:\widetilde{H}(\Omega^{\ast}) \longrightarrow \widetilde{H}(\Omega)$ such that $P^\ast f|_{\Omega^\ast}=f$.
\end{assumption}
 It is worth mentioning that another important ingredient in the controllability of problem (\ref{eq6.1.1}) (below) is the trace regularity of the conormal derivative $\partial_\nu u$ on $\partial \Omega^{\ast}$. This is obtained by using  Theorem 2 in Tataru \cite{4} on each smooth component $\Gamma_i^{\ast}$ of the boundary $\partial \Omega^{\ast}$ such that
$\partial\Omega^{\ast}=\overset{l}{\underset{i=1}\bigcup} \Gamma_i^{\ast}$. Note that the whole boundary $\partial\Omega^{\ast}$  must lie strictly inside   some $\mathcal{O}_k$, ~$k\in \{1,\cdots, N_0\}$ or $ \Omega^{\ast}$ must contain the set $\mathcal{O}_{N_0}$ properly. This is required since the metric $g_{ij}^{(k)}$ associated to the operator $P$ chances on each $\mathcal{O}_k$ and also in $\Omega \backslash \mathcal{O}_{N_0}$. Below, (see Figure 1) we present some favorable geometries for $\Omega^\ast$, where the boundary of $\Omega^\ast$ is bold dotted and $N_0=2$.

In order to verify how to apply the Theorem 2 in \cite{4}, let $\Gamma_i^{\ast}$ be, for some $i\in \{1,\cdots,l\}$ a generic component of the boundary $\partial\Omega^{\ast}$ and let us define $\Sigma_i^{\ast}:= \Gamma_i^{\ast} \times (0,T)$. Setting $L:= (\partial_{t}^{2}+ P)$, from (\ref{eq6.1.1}) one has $Lu=0$ in $\Omega^\ast \times (0,T)$ and we shall prove that $\widetilde u\in H^1_{loc}(\Theta)$, for any open set $\Theta$ of $\Omega \times \mathbb{R}_t \subset \mathbb{R}^n_x \times \mathbb{R}_t$, where $\widetilde u$ represents the extension of $u$ in $\Omega$ by considering zero out of $\Omega_{\delta}^{\ast}$. Let $\phi \in C_0^{\infty}(\Theta)$ be a cut off function such that $\phi=1$ in a neighbourhood of $\Sigma_i^{\ast}$ in $\mathbb{R}^n_x \times \mathbb{R}_t$ with $\supp(\phi) \subset \mathcal{O}_k \times \mathbb{R}$ for some $k\in \{1,\cdots, N_0\}$ (or $\supp(\phi)\subset (\Omega\backslash \mathcal{O}_{N_0}) \times \mathbb{R}$ properly). Thus, $L (\phi u) = \phi Lu + [L,\phi]u= [L,\phi]u\in L^2_{loc}(\Theta)$, since $Lu=0$ and $[L,\phi]$ has order $1$. So that from Theorem 2 in \cite{4} we deduce that $\partial_\nu (\phi u)\in L_{loc}^2( \Gamma_i^{\ast} \times \mathbb{R}_t )$, from which we conclude that $\partial_\nu u\in L_{loc}^2(\Sigma_i^{\ast})$ as desired. Pasting these traces we can define the desired control in $L^2(\partial \Omega^* \times ]0,T[)$.

For more complex geometries as those considered in Figure 2 we have to assume the following hypothesis:
\begin{assumption}\label{Tataru}
	If $\partial_{t}^2\widetilde{u}+P\widetilde{u} \in L_{loc}^{2}(\Omega^\ast \times ]0,T[))$ then $\partial_\nu \widetilde u $ is square integrable in each smooth part of $\partial \Omega^* \times ]0,T[$.
\end{assumption}	
\begin{figure}[H]
	
	\subfigure{\includegraphics[width=5cm]{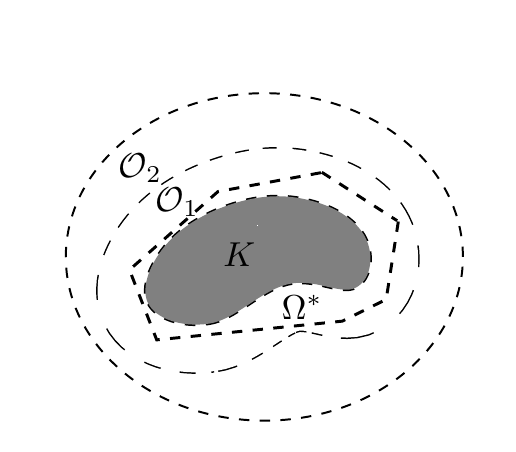}}
	\hfil
	\subfigure{\includegraphics[width=5cm]{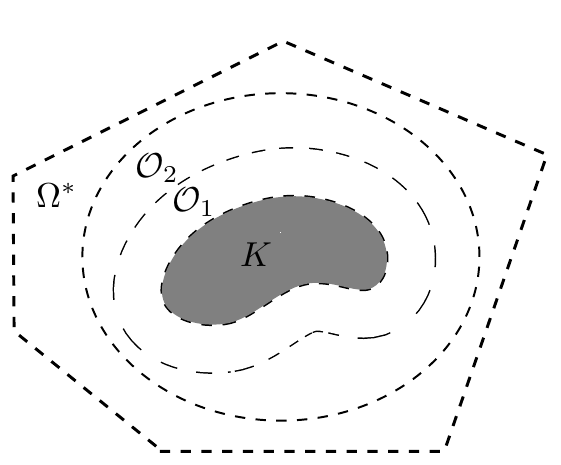}}
	
	\caption{Admissible geometries for $\Omega^\ast$. The control is located on the whole boundary $\partial \Omega^{\ast} \times (0,T)$.}
\end{figure}
\begin{figure}[H]
	
	\subfigure{\includegraphics[width=5cm]{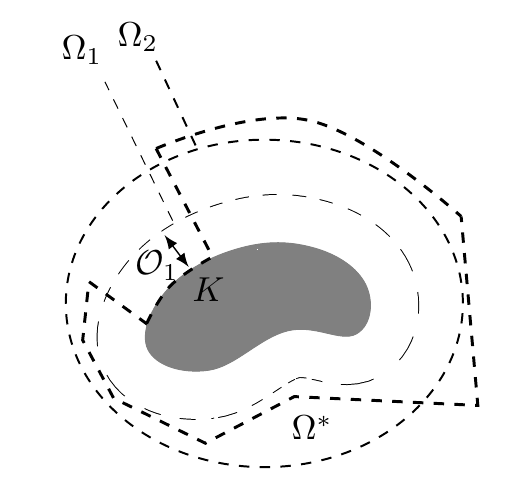}}
	
	\caption{Admissible complex geometry for $\Omega^\ast$. In this case $\partial_\nu u \in L^2(\partial \Omega^{\ast})$ is assumed because pieces of the boundary cross different sets. }
\end{figure}

Our main Theorem reads as follows:
\begin{theorem}\label{theo4.1.1}
	Then, under the Assumptions \ref{ass1}, \ref{ass2} and \ref{Tataru}, there exist $T>0$ and a control $g \in L^2(\partial \Omega^\ast \times (0,T))$, such that for every pair $(f_1,f_2) \in \widetilde{H}(\Omega^\ast) \times L^2(\Omega^\ast)$ the solution $u \in H^{1}_{loc}(\Omega^\ast \times (0,T))$ of the problem
	\begin{equation}
	\begin{cases}\label{eq6.1.1}
	(\partial_{t}^{2}+ P)u=0 &~\hbox{ in }~ \Omega^\ast \times (0,T), \\
	B^{\ast}u = g &~\hbox{ on }~\partial \Omega^\ast \times (0,T),\\
	u = 0 &~\hbox{ on }~\partial \Omega \times (0,T),\\
	u(x,0)=f_1(x), ~\partial_t u(x,0)=f_2(x) &~\hbox{ on }~ \Omega^\ast,
	\end{cases}
	\end{equation}
	satisfies $u(\cdot,T)=\partial_t u(\cdot,T)=0$ on $\Omega^\ast$, where $B^\ast u=\alpha u+\beta \partial_\nu u$, with $\alpha, \beta \in \mathbb{R}$ and $\alpha^2+\beta^2 \neq 0$.
\end{theorem}


We observe that in both configurations of Figure 1, the control $g$ is located on the whole boundary $\partial \Omega^\ast \times (0,T)$. However, it is possible to construct certain geometries letting a piece of $\partial \Omega^\ast \times (0,T)$ without control when $\Omega^{\ast}$ is properly contained in $\mathcal{O}_1$ and the  piece of $\partial \Omega^\ast \times (0,T)$ without control is precisely part of the boundary of $\partial K \times (0,T)$ according to Figure 3.

\begin{figure}[H]
	
	\subfigure{\includegraphics[width=5cm]{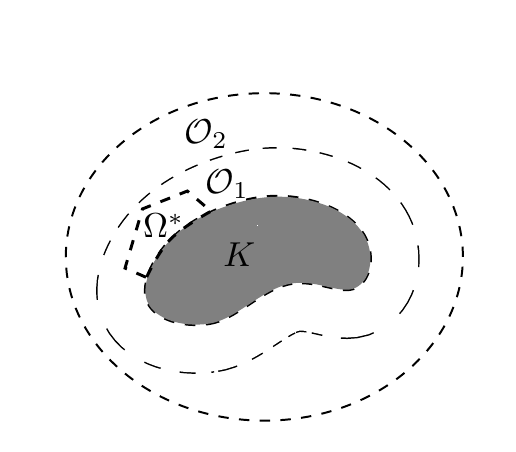}}
	\hfil
	
	\caption{Admissible geometry for $\Omega^\ast$. There is no control in the portion of $\partial \Omega^{\ast}$ which coincides with the boundary of $\partial K$.}
\end{figure}

An illustrative example of the aforementioned situation occurs in case of transmission problems in bounded domains as considered in Cardoso and Vodev \cite{23}. Indeed, for this purpose let $\Omega_1'\subset \Omega_2' \subset \cdots \subset \Omega_{m+1}' \subset \mathbb{R}^n$ ;~ $m\geq 1$, ~$n\geq 2$, be bounded, strictly convex domains with
smooth boundaries $\Gamma_k' = \partial \Omega_k'$;~ $\Gamma_k' \cap \Gamma_{k+1}'=\emptyset$. Let also $\Omega_0' \subset \Omega_1'$ be a bounded domain with smooth boundary $\Gamma_0' = \partial \Omega_0'$ such that $\mathbb{R}^n \backslash \Omega_0'$ is connected. Let us consider the following mixed boundary value problem:
\begin{equation}\label{TP}
\left\{
\begin{aligned}
& (\partial_t^2 - c_k^2\Delta) u_k =0 ~\hbox{ in } (\Omega_k'\backslash \Omega_{k-1}')\times (0,T);~k=1,\cdots,m+1,\\
&B u_1=0 ~\hbox{ on } \Gamma_0' \times (0,T),\\
&u_k = u_{k+1}; ~\partial_\nu u_k = \partial_\nu u_{k+1} ~\hbox{ on } \Gamma_k' \times (0,T),~k=1,\cdots, m,\\
&u_k(x,0) = f_k(x);~\partial_t u_k(x,0)=g_k(x);~k=1,\cdots, m+1,
\end{aligned}
\right.
\end{equation}
$\partial_\nu$ denotes the normal derivative to the boundary, $B=Id$ (Dirichlet boundary conditions) or $B=\partial_\nu$ (Neumann boundary conditions), $c_k$ are constants satisfying
\begin{eqnarray}\label{rel velocities}
c_1> c_2 > \cdots > c_{m+1}>0.
\end{eqnarray}

Equation (\ref{TP}) describes the propagation of acoustic waves in different
media with different speeds $c_k$, $k=1, \cdots ,m+1$, which do not penetrate into $\Omega_0'$. The following crucial assumption is also necessary:
\begin{assumption}\label{main assump}
Every generalized ray in $\Omega_1' \backslash \Omega_0'$ hits the boundary $\Gamma_1'$.
\end{assumption}
Clearly, Assumption \ref{main assump} is fulfilled if $\Omega_0'$ is strictly convex. However, the class of the domains for which Assumption $\ref{main assump}$
is satisfied is much larger than the class of strictly convex domains. Setting
$$H:= \bigoplus_{k=1}^m L^2\left(\Omega_k' \backslash \Omega_{k-1}', c_k^{-2} \,dx\right),$$
and assuming that (\ref{rel velocities}) and Assumption \ref{main assump} hold, one has the following very useful result regarding problem:
\begin{equation}\label{AP}
\left\{
\begin{aligned}
& (\partial_t^2 u - c_k^2\Delta) u_k =0 ~\hbox{ in } (\Omega_k'\backslash \Omega_{k-1}')\times (0,+\infty);~k=1,\cdots,m+1,\\
& (\partial_t^2 - c_{m+1}^2\Delta) u_{m+1}=0 ~\hbox{ in } (\mathbb{R}^n \backslash \Omega_{k-1}')\times (0,+\infty),\\
&Bu_1=0 ~\hbox{ on } \Gamma_0' \times (0,T),\\
&u_k = u_{k+1}; ~\partial_\nu u_k = \partial_\nu u_{k+1} ~\hbox{ on } \Gamma_k' \times (0,+\infty),~k=1,\cdots, m,\\
&u_k(x,0) = f_k(x);~\partial_t u_k(x,0)=g_k(x);~k=1,\cdots, m+1,
\end{aligned}
\right.
\end{equation}
\begin{theorem}[Theorem 1.5 in Cardoso-Vodev \cite{23}]
Under the Assumption \ref{main assump} and assuming that (\ref{rel velocities}) holds, for every compact $K \subset \mathbb{R}^n \backslash \Omega_0'$, there exists a constant $C_K$ so that the solution $u=(u_1, \cdots, u_{m+1})$ of (\ref{AP})satisfies the estimate (for $t\gg1$)
\begin{eqnarray}\quad\label{main estimate}
||\nabla_x u(\cdot, t)||_{L^2(K)} + || \partial_t u(\cdot, t)||_{L^2(K)} \leq C_K p_0(t) \left(||\nabla_x u(\cdot,t)||_{L^2(K)} + ||\partial_t u(\cdot, 0)||_{L^2(K)} \right),
\end{eqnarray}
provided $\supp u(\cdot,0),~\supp \partial_t u(\cdot,0)\subset K$, where
\begin{equation*}
p_0(t) =\left\{
\begin{aligned}
& e^{-\gamma t} ~\hbox{ if } n \hbox{ is odd},\\
& t^{-n}~ \hbox{ if } n \hbox{ is even},
\end{aligned}
\right.
\end{equation*}
with a constant $\gamma>0$ independent of $t$.
\end{theorem}

Let us consider, according to aforementioned notation, $\Omega^\ast$ be a bounded domain of $\mathbb{R}^n$, with boundary $\partial \Omega^\ast$ piecewise smooth with no cuspidal points, such that $ \Omega^\ast \subset B_a \cap (\mathbb{R}^n\backslash \Omega_0')$.
We are interested in studying the controllability of the solutions of the mixed boundary value problem (\ref{eq6.1.1}) but now in connection with transmission problems. An easy structure to be considered is that one when $\Omega^{\ast}= \Omega_{m+1}'$. In this case the exact controllability problem reads as follows: to find a control $g\in L^2(\Gamma_{m+1}' \times (0,T))$ which drives the problem
\begin{equation}\label{CP}
\left\{
\begin{aligned}
& (\partial_t^2 u - c_k^2\Delta) u_k =0 ~\hbox{ in } (\Omega_k'\backslash \Omega_{k-1}')\times (0,T);~k=1,\cdots,m+1,\\
&B u_1=0 ~\hbox{ on } \Gamma_0' \times (0,T),\\
&u_k = u_{k+1}; ~\partial_\nu u_k = \partial_\nu u_{k+1} ~\hbox{ on } \Gamma_k' \times (0,T),~k=1,\cdots, m,\\
& B^{\ast} u_{m+1}=g~\hbox{ on }\Gamma_{m+1}'\times (0,T),\\
&u_k(x,0) = f_k(x);~\partial_t u_k(x,0)=g_k(x);~k=1,\cdots, m+1,
\end{aligned}
\right.
\end{equation}
to the state $u(T) = \partial_t u(T)=0$, with $u=(u_1, \cdots,u_{m+1})$.

The above case, although interesting, possesses a smooth boundary $\Gamma_{m+1}'$. The most interesting case occurs when $\Omega^{\ast}$ is a bounded set with boundary $\partial \Omega^\ast$ piecewise smooth with no cuspidal points and suitably accommodated inside the transmission zone. Note that, as before, $\partial \Omega^{\ast}$ must lie strictly in some $\Omega_k'$,~for $k\in \{1,\cdots, m\}$ or $\Omega^{\ast}$ must contain the set $\Omega_{m+1}'$ properly.  Please, find in Figures 4 and 5, some illustrations of favorable geometries for $m=2$.
\begin{figure}[H]
	\subfigure{\includegraphics[width=5cm]{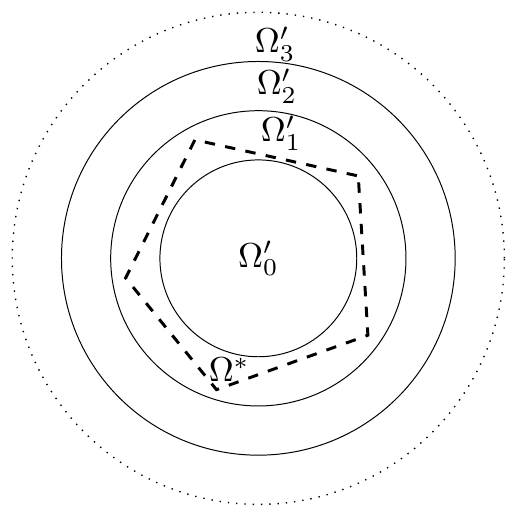}}
	\hfil
	\subfigure{\includegraphics[width=5cm]{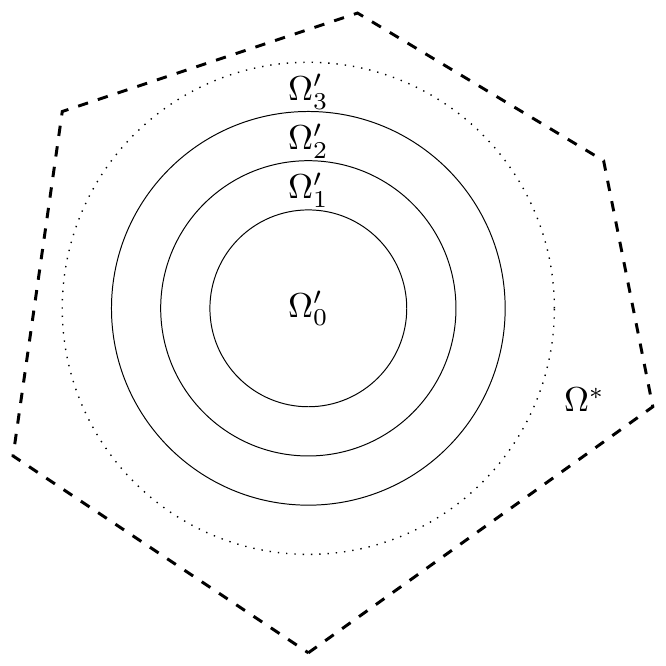}}
	
	\caption{Admissible geometries for $\Omega^\ast$.}
\end{figure}
\begin{figure}[H]
	\subfigure{\includegraphics[width=5cm]{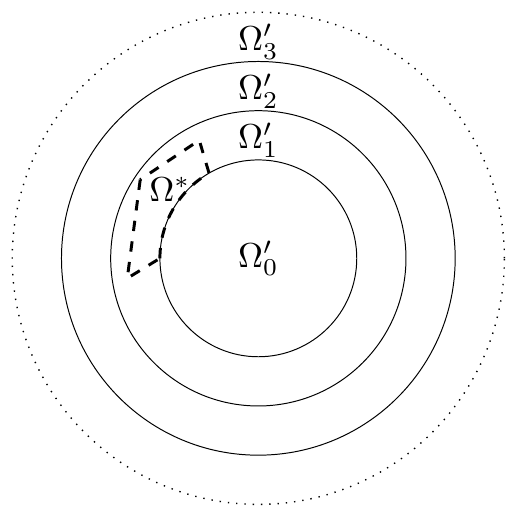}}
	\hfil
	
	\caption{Admissible geometry for $\Omega^\ast$.}
\end{figure}

In the case when there is no transmission of waves  (which corresponds to taking $m=0$ in the setting above), the controllability follows from the results in  Bardos,  Lebeau and  Rauch \cite{9}. In fact, in \cite{9}, a more general situation is studied, namely, $\Omega_1'$ is not necessarily strictly convex and the control $g$ is supposed to hold on a nonempty subset $\tilde{\Gamma}_1'$ of $\Gamma_1'$.
Then Assumption \ref{main assump} is replaced by the assumption that every generalized ray in $\Omega_1' \backslash \Omega_0'$ hits $\tilde{\Gamma}_1'$ at a nondiffractive point. The situation changes drastically in the case of transmission (which corresponds to taking $m\geq1$ in the setting above) due to the fact that the classical flow for this problem is much more complicated. Indeed, when a ray in $\Omega_{k+1}' \backslash \Omega_k'$ hits the boundary $\Gamma_k'$ (if $1\leq k\leq m$ or the boundary
$\Gamma_{k+1}'$ (if $0\leq k\leq m-1$), it splits into two rays-one staying in $\Omega_{k+1}' \backslash \Omega_k'$ and another entering into
$\Omega_k' \backslash \Omega_{k-1}'$ or $\Omega_{k+2}\backslash \Omega_{k+1}$, respectively. Consequently, there are infinitely many rays that do not reach
the boundary $\Gamma_{m+1}'$ where the dissipation is active. The condition (\ref{rel velocities}), however, guarantees that these rays carry a negligible amount of energy, and therefore (\ref{rel velocities}) is crucial for the controllability to hold (see Cardoso and Vodev \cite{23} and references therein for more details).

\subsection{Literature overview}

We start this subsection by quoting the pioneer article due to Russel \cite{3} in which, taking advantage of certain local decay rate estimates for the wave equation in the exterior of star-shaped regions (see the Scattering Theory of Lax and Phillips \cite{12}), he establishes the exact controllability of weak solutions of the wave equation by considering a Dirichlet control acting on the boundary $\partial \Omega$ of a bounded domain $\Omega$, where $\partial \Omega$ is assumed to be piecewise smooth.

Later on, exploiting ideas from \cite{3}, Lagnese \cite{5} proved the exact controllability of regular solutions of the wave equation by considering a boundary control of Dirichlet, Neumann or Robin type, posed in a bounded domain $\Omega$ with smooth boundary $\partial \Omega$. In this article Lagnese \cite{5} gave an affirmative answer for a certain class of hyperbolic operators by showing that the exact controllability can be achieved in any time $T$ which exceeds the diameter of $\Omega$.

It is worth mentioning other papers in connection with the techniques developed in \cite{5} and \cite{3} as, for instance  \cite{6}, \cite{7}. In \cite{6} the authors study the exact controllability for a class of hyperbolic linear partial differential equation with coefficients constants, which includes the Klein-Gordon equation, by considering piecewise smooth domains on the plane and boundary control of Robin type acting on the whole boundary. In \cite{7} the authors study the local asymptotic behavior of the solutions of the linear Klein–Gordon equation in a piecewise smooth domain $\Omega$. For this purpose, instead of using a suitable scattering theory for the associated problem, the authors obtained new local decay rate estimates by exploiting explicit formulas for the solution of the corresponding Cauchy problem. In addition, the authors use the local decay of energy to study the exact boundary controllability (Robin control) for the linear Klein-Gordon equation in piecewise smooth domains. Another interesting reference in the same spirit, now for linearly coupled wave equations, can be found in \cite{25}.

In \cite{10} Lions uses its Hilbert Uniqueness Method (HUM) and treats the control problem with initial data $L^2(\Omega) \times H^{-1}(\Omega)$, in which, only domains with smooth boundary are considered. The controllability for the wave equation in nonsmooth domains has been first studied in \cite{8}, by Grisvard. There, Grisvard uses HUM to study the wave equation in polygons and polyhedrons. Contrarily, Russell’s approach \cite{3} has been used to study control for wave equation with finite energy initial state in nonsmooth domains.

Another important paper which deals with nonsmooth domains is due to Burq \cite{17}. The paper makes references to the exact controllability of the wave equation with Dirichlet boundary conditions in a bounded corner open subset $\Omega$ of $\mathbb{R}^2$. Under suitable hypotheses on the regularity of $\Omega$, the condition of geometrical control introduced by C. Bardos, G. Lebeau and J. Rauch \cite{9} is generalized. Via some results on the propagation of the $H^1$-singularities, it is mainly shown that the geometrical control condition is a sufficient one for the exact boundary controllability of the wave equation in $\Omega$.

Although there exists a truly long bibliography regarding the controllability of the wave equation with constant coefficients, very few has been published for the wave equation with variable coefficients. Among them we would like to mention the works of Lasiecka, Triggiani and Yao \cite{22}, \cite{21}, \cite{31}, \cite{30}.

The authors consider a general second-order hyperbolic equation defined on an open bounded domain $\Omega$ with smooth boundary $\partial \Omega$ of class $C^2$ with variable coefficients in both the elliptic principal part and in the first-order terms as well. Initially, no boundary conditions B.C. are imposed. Their main result (Theorem 3.5) is a reconstruction, or inverse, estimate for solutions under suitable conditions on the coefficients of the principal part, the $H^1(\Omega) \times L^2(\Omega)$-energy at time $t = T$, or at time $t = 0$, is dominated by the $L^2(\Sigma)$-norms of the boundary traces $\partial \nu_{A}(\cdot)$ and $\partial_t(\cdot) $ modulo an interior lower-order term. Once homogeneous B.C. are imposed, their results yield, under a
uniqueness theorem, needed to absorb the lower-order term, continuous observability estimates for both the Dirichlet and Neumann case, with an explicit, sharp observability time; hence, by duality, exact controllability results. Furthermore, no artificial geometrical conditions are imposed on the controlled part of the boundary in the Neumann case.

In contrast with existing literature, the first step of their method employs a Riemannian geometry approach to reduce the original variable coefficient principal part problem to a problem on an appropriate Riemannian manifold determined by the coefficients of the principal part, where the principal part is the d'Alembertian. In their second step, they employ explicit Carleman estimates at the differential level to deal with the variable first-order energy level terms. In their third step, the authors employ microlocal analysis yielding a sharp trace estimate, to remove artificial geometrical conditions on the controlled part of the boundary, in the Neumann case.

It is worth mentioning the work of Burq \cite{18}, which deals with variable coefficients as well and in which the author considers the problem of the exact boundary controllability of the linear wave equation with Dirichlet control. Using the so-called H-measures or microlocal defect measures introduced by L. Tartar and P. G\'erard, the author extends the results by C. Bardos, G. Lebeau and J. Rauch \cite{9} that provide sufficient and almost necessary conditions for the exact controllability. The main contribution of this work consists in weakening the $C^\infty$ assumptions of \cite{9} on the regularity of the domain and of the coefficients. Indeed, the author proves that the same results hold when the domain is of class $C^3$ and the coefficients of the elliptic operator involved in the wave equation are of class $C^2$.

Finally, we would like to mention the important papers in connection with the exact controllability of transmission problems associated with the wave equation. The question of boundary controllability in problems of transmission has been considered by several authors. In particular Lions \cite{10} considered the system in the special case of two wave equations, namely,
\begin{eqnarray*}
\begin{cases}
 \partial_t^2 w_1 - a_1 \Delta w_1 =0 \quad \hbox{ in }\Omega_1 \times (0,T),\\
 \partial_t^2 w_2 - a_2 \Delta w_2 =0 \quad \hbox{ in }\Omega_2 \times (0,T),
\end{cases}
\end{eqnarray*}
where $\Omega, \Omega_1$ are a bounded open connected sets in $\mathbb{R}^n$ with smooth boundaries $\Gamma$ and $\Gamma_1$ respectively such that $\overline{\Omega_1} \subset \Omega$ and $\Omega_2:= \Omega\backslash \Omega_1$ whose boundary is $\Gamma_2 := \Gamma \cup \Gamma_1$. Here, $a_i>0$ ($i=1,2$) and $\Delta$ is the ordinary Laplacian in $\mathbb{R}^n$,
\begin{eqnarray*}
&&w_2 = h \hbox{ on }\Sigma=\Gamma \times (0,T), h \hbox{ is the control},\\
&& w_1=w_2,~a_1 \partial_\nu w_1= a_2 \partial_\nu w_2\hbox{ on }\partial \Omega_i,~i=1,2,\\
&& w_i|_{t=0}= \partial_t w_i|_{t=0}=0 \hbox{ on }\Omega_i,~i=1,2.
\end{eqnarray*}

Assuming that $\Omega_1$ is star shaped with respect to some point $x_0\in \Gamma_1$ and setting $\Gamma(x_0):=\{x\in \Gamma:(x-x_0)\cdot\nu(x)>0\}$, $\Sigma(x_0):=\Gamma(x_0) \times (0,T)$ where $\nu$ is the unit outer normal to $\Gamma$, Lions proved the exact boundary controllability assuming that $a_1>a_2$ and for $T>T(x_0)=2R(x_0)/\sqrt{a_2}$ and $R(x_0)=\max_{x\in \overline{\Omega_2}}|x-x_0|$.

Later on Lagnese \cite{26} generalized Lions \cite{10} by considering transmission problems for general second order linear hyperbolic systems having piecewise constant coefficients in a bounded, open connected set with smooth boundary and controlled through the Dirichlet boundary condition. It is proved that such a system is exactly controllable in an appropriate function space provided the interfaces where the coefficients have a jump discontinuity are all star shaped with respect to one and the same point and the coefficients satisfy a certain monotonicity condition.

Another interesting generalization of Lions \cite{10} has been considered by Liu \cite{27}. In this paper the author addresses the problem of control of the transmission wave equation. In particular, he considers the case where, due to total internal reflection of waves at the interface, the system may not be controlled from exterior boundaries. He shows that such a system can be controlled by introducing both boundary control along the exterior boundary and distributed control near the transmission boundary and give a physical explanation why the additional control near the transmission boundary might be needed for some domains.

To end this subsection we would like to quote the papers due to Nicaise \cite{28}, \cite{29} in which the author discusses the problem of exact controllability of networks of elastic polygonal membranes. The individual membranes are assumed to be coupled either at a vertex or along a whole common edge.  The author then derives energy estimates for regular solutions, which are then, by transposition, extended to weak solutions. As usual, direct and inverse energy inequalities of the sort shown establish a norm equivalence on a certain space (classically named $F$), the completion of which is the space in which the HUM-principle of Lions works. The space $F'$ then contains the null-controllable initial data. This space is weak enough to correspond to $L^2$-boundary controls along exterior edges satisfying sign conditions with respect to energy multipliers, to such controls along Dirichlet-edges, and, more importantly, to $H^1$-vertex controls at those vertices which are responsible for severe singularities. The corresponding solutions, for $(u_0,u_1)\in H\times V'$ with rather weak regularity $(C(0,T,D(A)'))$, are then shown to be null-controllable in a canonical finite time.

Another very nice paper that we would like to quote is the work of Miller \cite{24}, which although not related to controllability is very closed to the subject of investigation . This article deals with the propagation of high-frequency wave solutions to the scalar wave equation and to the Schr\"odinger equation. The results are formulated in terms of semiclassical measures (Wigner measures). The propagation is across a sharp interface between two inhomogeneous media. The author proves a microlocal version of Snell-Descartes's law of refraction which includes diffractive rays. Moreover, a radiation phenomenon for density of waves propagating inside an interface along gliding rays is illustrated. The measures of the traces of the solutions of the corresponding partial differential equations enable the author to derive some propagation properties for the measure of the solutions.

\subsection{Novel contribution of this work}

The primary goal of this article is to design a unified framework for boundary control theory associated to generalized wave equations (including the transmission problem admitting several zones of transmission). The novel features offered here are:

\begin{itemize}

\item The method presented allows us to give a unified form that simultaneously accommodates
domains with nonsmooth boundary (the most interesting case) or smooth boundary as well by considering the control of Dirichlet, Neuman or Robin type for generalized wave equations. In contrast, most of the currently available results on exact boundary controllability focus on either just smooth boundary or just nonsmooth ones.

\item In the context of controllability theory for wave equations with variable coefficients or even for constant coefficients, this paper is
the first to consider the case of the exact controllability from the boundary to generalized wave equations including the particular case of the transmission problems subject to several zones of transmission in contrast with the previous literature which takes into account just two zones of transmission.

\item It is worth mentioning that the boundary stabilization to problem \ref{CP} has been studied by Cardoso and Vodev in \cite{23}, but the boundary exact controllability  still remained an open problem.
\end{itemize}

\subsection{Outline of the arguments}

The method presented here is an extension of the pioneers works \cite{5}, \cite{3} that take advantage of decay rate estimates of the local energy obtained by scattering theory. While at that time when those papers \cite{5}, \cite{3} were written there were few results in this direction (\cite{12} and references therein), nowadays we have a wide assortment of nice results as in the works \cite{11}, \cite{23}, \cite{19}, \cite{1}, \cite{13}, \cite{2},  \cite{14}, and references therein.

Our special interest comes from the  work of Vodev \cite{1}, which extends previous results of the literature (see \cite{12}) regarding the uniform decay of local energy of the wave equation to more general perturbations (including the transmission problem) showing that any uniform decay of the local energy implies that it must decay like $\mathcal{O}(t^{-2n})$, ~$t \gg 1$, being the time and $n$ being the space dimension. As a particular case of the scattering theory obtained in previous studies we can mention the Theorem 1.5 of the most recent paper of the authors Cardoso and Vodev \cite{23}.

Finally we would like to observe that while in \cite{3} just the Dirichlet control has been considered for weak solutions of the wave equation with constant coefficients, in the present paper the control can be of Dirichlet, Neumann or Robin type also for weak solutions of the generalized wave equation. The two latest ones are much more delicate since it is not clear that the trace of the normal derivative of solutions of the generalized wave equation belong to $L^2$ of the lateral boundary of the domain.

In this direction the result obtained by Tataru \cite{4} regarding the regularity of boundary traces of the wave equation plays an essential role.  In the same spirit it is worth mentioning the paper due to Kim \cite{40} where the author is particularly interested in the regularity of those controls that can be obtained from Huyghen's principle for bounded convex domains of odd dimension and from an extension-inversion principle for general dimensions. He uses microlocal analysis to establish a regularity result for general second-order hyperbolic partial differential operators in an open domain of $\mathbb{R}^{n+1}$ (including the half-space). The result is then applied to the above-mentioned controllability problem in order to obtain trace regularity results, which in turn provide regularity results for the controls on an entire scale of "energy-spaces''.
Note that in \cite{5} also control of Neumann or Robin type were considered. However, for this purpose, regular solutions were considered, which is not the case in the present paper.

It is worth mentioning that the presence of the coefficients in the wave operator, as considered in the present paper, makes the analysis much more refined in terms of the rays of the geometrical optics.

In addition, since we are working in the exterior of an obstacle a nontrapping metric is crucial. While in the trapping case logarithmic local decay rate estimates can be obtained the controllability is no longer expected, at least for smooth boundaries, since it hurts severely the laws of the geometrical optics due to Bardos, Lebeau and Rauch \cite{9}.

From the above, the nice and old method introduced by Russel \cite{3} combined with a sharp scattering theory as in \cite{23}, \cite{19}, \cite{1}, \cite{13}, \cite{2},  \cite{14} and a powerful result of regularity of traces of the wave equation (or hyperbolic equations in general) as considered in Tataru \cite{4} are the main ingredients for treating the exact controllability of hyperbolic equations from the boundary posed in general domains.

\section{Proof of the main result}

We begin this section by developing some results from the Theorem \ref{theo1.1.1}, which are the fundamental ingredients to obtain the exact boundary controllability of the generalized wave equations.
\begin{lemma}\label{lem1.1.1}
Let $\Omega^\ast \subset \mathbb{R}^n$  as defined in Remark \ref{Omegastar}. Then, there exists a bounded linear operator $E_1:\widetilde{H}(\Omega^\ast)\longrightarrow \widetilde{H}(\Omega)$, such that for each $f \in \widetilde{H}(\Omega^\ast)$ we have that $E_1f|_{\Omega^\ast}=f$, $\supp E_1f\subset \Omega^{\ast}_{\delta}$ and $\|E_1f\|_{\widetilde{H}(\Omega)}\leq C\|f\|_{\widetilde{H}(\Omega^\ast)}$ for some constant $C>0$.
\end{lemma}
\begin{proof}
Let $\varphi \in C_{0}^{\infty}(\Omega)$ be a function such that $\varphi=1$ in $\Omega^\ast$ and $\varphi=0$ in $\Omega \setminus \Omega_{\frac{\delta}{2}}$. Define $E_1f=\varphi P^\ast f$, where $P^*$ is given by Assumption \ref{ass2} and $f \in \widetilde{H}(\Omega^\ast)$. Clearly, $E_1:\widetilde{H}(\Omega^\ast)\longrightarrow \widetilde{H}(\Omega)$ is linear, $E_1f=f$ in $\Omega^\ast$ and $\supp E_1f \subset \Omega_{\delta}^{\ast}$ for all $f \in \widetilde{H}(\Omega^\ast)$. Noting that, $E_1=M_{\varphi} \circ P^\ast$, where $M_\varphi:\widetilde{H}(\Omega)\longrightarrow \widetilde{H}(\Omega)$, is the multiplication operator, defined by $M_{\varphi}(\psi)=\varphi\psi$, the boundedness follows.
\end{proof}

\begin{theorem}\label{theo2.1.1}
Let $f_1 \in \widetilde{H}(\Omega^\ast)$ and $f_2 \in L^2(\Omega^\ast)$ be functions with norm not identically zero and suppose that $\supp f_j \subset \Omega^\ast$, $j=1,2$. Let $u$ be the solution of the problem
\begin{equation*}
\begin{cases}
(\partial_{t}^{2}+ P)u=0 &~\hbox{ in }~ \Omega \times \mathbb{R}, \\
u = 0 &~\hbox{ on }~\partial \Omega \times \mathbb{R},\\
u(x,0)=f_1(x), ~\partial_t u(x,0)=f_2(x) &~\hbox{ in }~  \Omega.
\end{cases}
\end{equation*}
Then, there exists a positive constant $C$, independent of $f_1$ and $f_2$, such that
\begin{equation}\label{eq7.1.1}
\frac{\|\nabla_xu(x,t)\|_{L^2(\Omega^\ast)}+\|\partial_tu(x,t)\|_{L^2(\Omega^\ast)}}{\|\nabla_x u(x,0)\|_{L^2(\Omega^\ast)}+\|\partial_{t}u(x,0)\|_{L^2(\Omega^\ast)}}\leq
\begin{cases}
Ce^{-\gamma t}, &~\hbox{ if }~ n ~\hbox{ is odd, }~  \\
Ct^{-n}, &~\hbox{ if }~ n ~\hbox{ is even, }~
\end{cases},
\end{equation}
for each $t>T_0$.
\end{theorem}
\begin{proof}
By elementary measure theory,
\begin{align*}
\|\nabla_xu(x,t)\|_{L^2(\Omega^\ast)}+\|\partial_tu(x,t)\|_{L^2(\Omega^\ast)} \leq & \|\nabla_xu(x,t)\|_{L^2(B_a \cap \Omega)}+\|\partial_tu(x,t)\|_{L^2(B_a \cap \Omega)}.
\end{align*}
Taking into account the supports of $f_1$ and $f_2$, we obtain
\begin{align*}
\|\nabla_x u(x,0)\|_{L^2(\Omega^\ast)}+\|\partial_{t}u(x,0)\|_{L^2(\Omega^\ast)}= \|\nabla_x u(x,0)\|_{L^2(B_a \cap\Omega)}+\|\partial_{t} u(x,0)\|_{L^2(B_a \cap\Omega)}.
\end{align*}
Therefore,
\begin{align*}
\frac{\|\nabla_xu(x,t)\|_{L^2(\Omega^\ast)}+\|\partial_tu(x,t)\|_{L^2(\Omega^\ast)}}{\|\nabla_x u(x,0)\|_{L^2(\Omega^\ast)}+\|\partial_{t}u(x,0)\|_{L^2(\Omega^\ast)}} \leq & \frac{\|\nabla_xu(x,t)\|_{L^2(B_a \cap \Omega)}+\|\partial_tu(x,t)\|_{L^2(B_a \cap \Omega)}}{\|\nabla_x u(x,0)\|_{L^2(B_a \cap\Omega)}+\|\partial_{t} u(x,0)\|_{L^2(B_a \cap\Omega)}}.
\end{align*}
From item iii) of Theorem \ref{theo1.1.1}, we obtain
\begin{align*}
\frac{\|\nabla_xu(x,t)\|_{L^2(\Omega^\ast)}+\|\partial_tu(x,t)\|_{L^2(\Omega^\ast)}}{\|\nabla_x u(x,0)\|_{L^2(\Omega^\ast)}+\|\partial_{t}u(x,0)\|_{L^2(\Omega^\ast)}} \leq& \frac{\|\nabla_xu(x,t)\|_{L^2(B_a \cap \Omega)}+\|\partial_tu(x,t)\|_{L^2(B_a \cap \Omega)}}{\|\nabla_x u(x,0)\|_{L^2(B_a \cap\Omega)}+\|\partial_{t} u(x,0)\|_{L^2(B_a \cap\Omega)}}\\
\leq& \begin{cases}
Ce^{-\gamma t}, &~\hbox{ if }~ n ~\hbox{ is odd, }~  \\
Ct^{-n}, &~\hbox{ if }~ n ~\hbox{ is even, }~
\end{cases},
\end{align*}
for each $t>T_0$.
\end{proof}

\begin{theorem}\label{theo3.1.1}
Let $f_1 \in \widetilde{H}(\Omega^\ast)$ and $f_2 \in L^2(\Omega^\ast)$ be functions with norm not identically zero and suppose that $\supp f_j \subset \Omega^\ast$, $j=1,2$. Let $u$ the solution of the problem
	\begin{equation*}
	\begin{cases}
	(\partial_{t}^{2}+ P)u=0 &~\hbox{ in }~ \Omega \times \mathbb{R}, \\
	u = 0 &~\hbox{ on }~\partial \Omega \times \mathbb{R},\\
	u(x,T)=f_1(x), ~\partial_t u(x,T)=f_2(x) &~\hbox{ in }~  \Omega.
	\end{cases}
	\end{equation*}
	Then, there exists a positive constant $C$, independent of $f_1$ and $f_2$, such that
	\begin{equation}\label{eq8.1.1}
\frac{\|\nabla_xu(x,0)\|_{L^2(\Omega^\ast)}+\|\partial_tu(x,0)\|_{L^2(\Omega^\ast)}}{\|\nabla_x u(x,T)\|_{L^2(\Omega^\ast)}+\|\partial_{t}u(x,T)\|_{L^2(\Omega^\ast)}}\leq
	\begin{cases}
	Ce^{-\gamma T}, &~\hbox{ if }~ n ~\hbox{ is odd, }~  \\
	CT^{-n}, &~\hbox{ if }~ n ~\hbox{ is even. }~
	\end{cases},
	\end{equation}
	for each $T>T_0$.
\end{theorem}
\begin{proof}
Let $v$ be the solution of the problem
\begin{equation*}
	\begin{cases}
	(\partial_{\tau}^{2}+ P)v=0 &~\hbox{ in }~ \Omega \times \mathbb{R}, \\
	v = 0 &~\hbox{ on }~\partial \Omega \times \mathbb{R},\\
	v(x,0)=f_1(x), ~\partial_\tau v(x,0)=-f_2(x) &~\hbox{ in }~  \Omega.
	\end{cases}
	\end{equation*}
Applying the estimate \eqref{eq7.1.1} to $v$, we conclude that
\begin{equation}\label{eq9.1.1}
\frac{\|\nabla_xv(x,\tau)\|_{L^2(\Omega^\ast)}+\|\partial_\tau v(x,\tau)\|_{L^2(\Omega^\ast)}}{\|\nabla_x v(x,0)\|_{L^2(\Omega^\ast)}+\|\partial_{\tau} v(x,0)\|_{L^2(\Omega^\ast)}}\leq
\begin{cases}
Ce^{-\gamma \tau}, &~\hbox{ if }~ n ~\hbox{ is odd, }~  \\
C\tau^{-n}, &~\hbox{ if }~ n ~\hbox{ is even, }~
\end{cases},
\end{equation}
for $\tau>T_0$. Making $\tau=T-t$ in \eqref{eq9.1.1} and noting that $v(\cdot,T-t)=u(\cdot,t)$, which implies that $\partial_{\tau}v(\cdot,T-t)=-\partial_tu(\cdot,t)$, we obtain
\begin{equation}\label{eq10.1.1}
\frac{\|\nabla_xu(x,t)\|_{L^2(\Omega^\ast)}+\|\partial_tu(x,t)\|_{L^2(\Omega^\ast)}}{\|\nabla_xu(x,T)\|_{L^2(\Omega^\ast)}+\|\partial_{t}u(x,T)\|_{L^2(\Omega^\ast)}}\leq
\begin{cases}
Ce^{-\gamma (T-t)}, &~\hbox{ if }~ n ~\hbox{ is odd, }~  \\
C(T-t)^{-n}, &~\hbox{ if }~ n ~\hbox{ is even, }~
\end{cases},
\end{equation}
whenever $T-t>T_0$. Choosing $t=0$ in \eqref{eq10.1.1} we have the desired result.
\end{proof}
\begin{corollary}\label{cor1.1.1}
Let $f_1 \in \widetilde{H}(\Omega^\ast)$ and $f_2 \in L^2(\Omega^\ast)$ be functions with norm not identically zero and suppose that $\supp f_j \subset \Omega^\ast$, $j=1,2$. Let $u$ be the solution of the problem
	\begin{equation*}
	\begin{cases}
	(\partial_{t}^{2}+ P)u=0 &~\hbox{ in }~ \Omega \times \mathbb{R}, \\
	u = 0 &~\hbox{ on }~\partial \Omega \times \mathbb{R},\\
	u(x,0)=f_1(x), ~\partial_t u(x,0)=f_2(x) &~\hbox{ in }~
	\Omega.
	\end{cases}
	\end{equation*}
Then, there exists a positive constant $C$, independent of $f_1$ and $f_2$, such that
	\begin{equation}\label{eq11.1.1}
	\frac{\|\nabla_xu(x,t)\|_{L^2(\Omega_{\delta}^{\ast})}^{2}+\|\partial_tu(x,t)\|_{L^2(\Omega_{\delta}^{\ast})}^{2}}{\|\nabla_x u(x,0)\|_{L^2(\Omega_{\delta}^{\ast})}^{2}+\|\partial_tu(x,0)\|_{L^2(\Omega_{\delta}^{\ast})}^{2}}\leq
	\begin{cases}
	C'e^{-2\gamma t}, &~\hbox{ if }~ n ~\hbox{ is odd, }~  \\
	C't^{-2n}, &~\hbox{ if }~ n ~\hbox{ is even, }~
	\end{cases},
	\end{equation}
for each $t>T_0$.
\end{corollary}
\begin{proof}
Follows directly from Theorem \ref{theo2.1.1}.
\end{proof}

\begin{corollary}\label{cor2.1.1}
	Let $f_1 \in \widetilde{H}(\Omega^\ast)$ and $f_2 \in L^2(\Omega^\ast)$ be functions with norm not identically zero and suppose that $\supp f_j \subset \Omega^\ast$, $j=1,2$.Let $u$ the solution of the problem
\begin{equation*}
\begin{cases}
(\partial_{t}^{2}+ P)u=0 &~\hbox{ in }~ \Omega \times \mathbb{R}, \\
u = 0 &~\hbox{ on }~\partial \Omega \times \mathbb{R},\\
u(x,T)=f_1(x), ~\partial_t u(x,T)=f_2(x) &~\hbox{ in }~  \Omega.
\end{cases}
\end{equation*}
Then, there exists a positive constant $C$, independent of $f_1$ and $f_2$, such that
\begin{equation}\label{eq12.1.1}
\frac{\|\nabla_xu(x,0)\|_{L^2(\Omega_{\delta}^{\ast})}^{2}+\|\partial_tu(x,0)\|_{L^2(\Omega_{\delta}^{\ast})}^{2}}{\|\nabla_x u(x,T)\|_{L^2(\Omega_{\delta}^{\ast})}^{2}+\|\partial_{t}u(x,T)\|_{L^2(\Omega_{\delta}^{\ast})}^{2}}\leq
\begin{cases}
C'e^{-2\gamma T}, &~\hbox{ if }~ n ~\hbox{ is odd, }~  \\
C'T^{-2n}, &~\hbox{ if }~ n ~\hbox{ is even, }~
\end{cases},
\end{equation}
for each $T>T_0$.
\end{corollary}
\begin{proof}
Let $v$ be the solution of the problem
\begin{equation*}
\begin{cases}
(\partial_{\tau}^{2}+ P)v=0 &~\hbox{ in }~ \Omega \times \mathbb{R}, \\
v = 0 &~\hbox{ on }~\partial \Omega \times \mathbb{R},\\
v(x,0)=f_1(x), ~\partial_\tau v(x,0)=-f_2(x) &~\hbox{ in }~  \Omega.
\end{cases}
\end{equation*}
Applying the estimate \eqref{eq11.1.1} to $v$, we obtain
\begin{equation}\label{eq13.1.1}
\frac{\|\nabla_xv(x,\tau)\|_{L^2(\Omega_{\delta}^{\ast})}^{2}+\|\partial_\tau v(x,\tau)\|_{L^2(\Omega_{\delta}^{\ast})}^{2}}{\|\nabla_xv(x,0)\|_{L^2(\Omega_{\delta}^{\ast})}^{2}+\|\partial_{\tau} v(x,0)\|_{L^2(\Omega_{\delta}^{\ast})}^{2}}\leq
\begin{cases}
C'e^{-2\gamma \tau}, &~\hbox{ if }~ n ~\hbox{ is odd, }~  \\
C'\tau^{-2n}, &~\hbox{ if }~ n ~\hbox{ is even, }~
\end{cases},
\end{equation}
for $\tau>T_0$. Making $\tau=T-t$ in \eqref{eq13.1.1} and noting that $v(\cdot,T-t)=u(\cdot,t)$, we conclude that
\begin{equation}\label{eq14.1.1}
\frac{\|\nabla_xu(x,t)\|_{L^2(\Omega_{\delta}^{\ast})}^{2}+\|\partial_tu(x,t)\|_{L^2(\Omega_{\delta}^{\ast})}^{2}}{\|\nabla_x u(x,T)\|_{L^2(\Omega_{\delta}^{\ast})}^{2}+\|\partial_{t}u(x,T)\|_{L^2(\Omega_{\delta}^{\ast})}^{2}}\leq
\begin{cases}
C'e^{-2\gamma (T-t)}, &~\hbox{ if }~ n ~\hbox{ is odd, }~  \\
C'(T-t)^{-2n}, &~\hbox{ if }~ n ~\hbox{ is even, }~
\end{cases},
\end{equation}
whenever $T-t>T_0$. Choosing $t=0$ in \eqref{eq14.1.1} we have the desired result.
\end{proof}

With the above results, we can construct the operators necessary to obtain the exact boundary controllability.

Let $u$ be the solution of the problem
\begin{equation}\label{eq15.1.1}
\begin{cases}
(\partial_{t}^{2}+ P)u=0 &~\hbox{ in }~ \Omega \times \mathbb{R}, \\
u = 0 &~\hbox{ on }~\partial \Omega \times \mathbb{R},\\
u(x,0)=f_1(x), ~\partial_t u(x,0)=f_2(x) &~\hbox{ in }~  \Omega,
\end{cases}
\end{equation}
where $(f_1,f_2)\in \widetilde{H}(\Omega)\times L^2(\Omega)$ and $\supp f_j \subset \Omega^\ast$. Now, for $t\neq0$, we define
\begin{equation*}
S_t:\widetilde{H}(\Omega_{\delta}^{\ast})\times L^2(\Omega_{\delta}^{\ast}) \longrightarrow \widetilde{H}(\Omega)\times L^2(\Omega),
\end{equation*}
is given by $S_t(u(\cdot,0),\partial_{t}u(\cdot,0))=(u(\cdot,t),\partial_{t}u(\cdot,t))$, where $u$ is the solution of the problem \eqref{eq15.1.1}. From the linearity of the operator $P$, it follows that $S_t$ is linear. Taking into account the supports of $f_j$ with $j=1,2$, we have that
\begin{align*}
\|S_t(u(\cdot,0),\partial_{t}u(\cdot,0))\|_{\widetilde{H}(\Omega)\times L^2(\Omega)}^{2} =& \| u(\cdot,t)\|_{\widetilde{H}(\Omega)}^{2}+\|\partial_{t} u(\cdot,t)\|_{L^2(\Omega)}^{2} \\
=&\| u(\cdot,0)\|_{\widetilde{H}(\Omega)}^{2}+\|\partial_{t} u(\cdot,0)\|_{L^2(\Omega)}^{2} \\
=&\| u(\cdot,0)\|_{\widetilde{H}(\Omega_{\delta}^{\ast})}^{2}+\|\partial_{t} u(\cdot,0)\|_{L^2(\Omega_{\delta}^{\ast})}^{2} \\
=& \|(u(\cdot,0),\partial_{t}u(\cdot,0))\|_{\widetilde{H}(\Omega_{\delta}^{\ast})\times L^2(\Omega_{\delta}^{\ast})}^{2},
\end{align*}
which shows that $S_t$ is bounded.

Finally, the operator $S_0$ extends $f_1 \in \widetilde{H}(\Omega^\ast)$ and $f_2 \in L^2(\Omega^\ast)$ respectively, by zero outside $\Omega^\ast$. Which has the same characteristics of $S_t$, for $t\neq 0$.

Similarly, we define the operator $$ S_{T}^{\ast}:\widetilde{H}(\Omega_{\delta}^{\ast})\times L^2(\Omega_{\delta}^{\ast})\longrightarrow \widetilde{H}(\Omega)\times L^2(\Omega),$$
defined by $S_{T}^{\ast}(u(\cdot,T),\partial_{t}u(\cdot,T))=(u(\cdot,0),\partial_{t}u(\cdot,0))$ where $u$ is the solution of the problem
	\begin{equation*}
\begin{cases}
(\partial_{t}^{2}+ P)u=0 &~\hbox{ in }~ \Omega \times \mathbb{R}, \\
u = 0 &~\hbox{ on }~\partial \Omega \times \mathbb{R},\\
u(x,T)=f_1(x), ~\partial_t u(x,T)=f_2(x) &~\hbox{ in }~  \Omega,
\end{cases}
\end{equation*}
with $(f_1,f_2)\in \widetilde{H}(\Omega)\times L^2(\Omega)$ and $\supp f_j \subset \Omega^\ast$. The operator $S_{T}^{\ast}$ is linear and bounded.

Now, we present a result which together with the Trace Theorem, due to Tataru \cite{4}, allows us to solve the boundary control problem for the equation studied in this paper. This approach has been introduced by D.L.Russell in \cite{3}, to solve control problems for the wave equation.

\begin{proof}{(\bf Theorem \ref{theo4.1.1})} Let $\Omega_{\delta}^{\ast}$, with $\delta>0$, and the operator $E_1:\widetilde{H}(\Omega^{\ast}) \longrightarrow \widetilde{H}(\Omega)$ be as defined in Lemma \ref{lem1.1.1}. Let $E_0:L^2(\Omega^{\ast}) \longrightarrow L^2(\Omega)$ be the operator, which extends $w_1 \in L^2(\Omega^\ast)$ to a function $E_0w_1 \in L^2(\Omega)$ with support in $\Omega_{\delta}^{\ast}$. The operator $E:\widetilde{H}(\Omega^\ast)\times L^2(\Omega^\ast) \longrightarrow \widetilde{H}(\Omega)\times L^2(\Omega)$, defined by $E(w_0,w_1)=(E_1w_0,E_0w_1)$ is linear and continuous. Furthermore, $\supp E_1w_0,\supp E_0 w_1 \subset \Omega_{\delta}^{\ast}$ for all $(w_0,w_1) \in \widetilde{H}(\Omega^\ast)\times L^2(\Omega^\ast)$. Let $(w_0,w_1)\in \widetilde{H}(\Omega^\ast)\times L^2(\Omega^\ast)$ and $w_\delta$ be the solution of the problem
\begin{equation}
\begin{cases}\label{eq16.1.1}
(\partial_{t}^{2}+ P)w_\delta=0 &~\hbox{ in }~ \Omega \times \mathbb{R}, \\
w_\delta=0 & ~\hbox{ in } \partial \Omega,\\
w_\delta(x,0)=E_1w_0(x), ~\partial_t w_\delta(x,0)=E_0w_1(x) &~\hbox{ in }~ \Omega.
\end{cases}
\end{equation}
Let $T>T_0$ be a number to be chosen later and $\varphi \in C_{0}^{\infty}(\Omega)$ such that $\varphi=1 \in \Omega_{\frac{\delta}{2}}^{\ast}$ and $\varphi=0$ in the complement of $\Omega_{\frac{3\delta}{4}}^{\ast}$. Note that,
\begin{equation*}
(\varphi w_\delta(\cdot,T),\varphi\partial_tw_\delta(\cdot,T))\in \widetilde{H}(\Omega^{\ast}_{\delta})\times L^2(\Omega^{\ast}_{\delta}).
\end{equation*}
Let $z$ be the solution of the problem
\begin{equation}
\begin{cases}\label{eq17.1.1}
(\partial_{t}^{2}+ P)z=0 &~\hbox{ in }~ \Omega \times \mathbb{R}, \\
z=0 & ~\hbox{ on }~\partial \Omega, \\
z(\cdot,T)=\varphi w_\delta(\cdot,T), ~\partial_t z(\cdot,T)=\varphi \partial_{t} w_\delta(\cdot,T) &~\hbox{ on }~ \Omega.
\end{cases}
\end{equation}
Using the operators $S_T, \ S_{T}^{\ast}, \ M_\varphi$ and $E$, we can write
\begin{align*}
(\varphi w_\delta(\cdot,T),\varphi\partial_tw_\delta(\cdot,T) ) =& M_\varphi S_{T} (E_1w_0,E_0w_1) \\
=&M_\varphi S_{T}E(w_0,w_1)
\end{align*}
and
\begin{align*}
(z(\cdot,0),\partial_tz(\cdot,0) =& S_{T}^{\ast}(z(\cdot,T),\partial_tz(\cdot,T))\\
=&S_{T}^{\ast}(\varphi w_\delta(\cdot,T),\varphi\partial_{t}w_\delta(\cdot,T)) \\
=& S_{T}^{\ast}M_\varphi S_{T} E(w_0,w_1).
\end{align*}
Define
\begin{equation*}
\widetilde{u}=w_\delta-z.
\end{equation*}
Observe that $\widetilde{u}$ solves the problem
\begin{equation}
\begin{cases}\label{eq18.1.1}
(\partial_{t}^{2}+ P)\widetilde{u}=0 &~\hbox{ in }~ \Omega \times \mathbb{R}, \\
\widetilde{u}=0 & ~\hbox{ on }~ \partial \Omega, \\
\widetilde{u}(\cdot,0)=E_1w_0-z(\cdot,0), ~\partial_t \widetilde{u}(x,0)=E_0w_1-\partial_{t}z(\cdot,0) &~\hbox{ in }~ \Omega.
\end{cases}
\end{equation}
In addition, the following conditions are verified
\begin{equation*}
\begin{cases}
\widetilde{u}(\cdot,T)=w_\delta(\cdot,T)-\varphi w_\delta(\cdot,T)=(1-\varphi(\cdot))w_\delta(\cdot,T) \\
\partial_{t}\widetilde{u}(\cdot,T)=\partial_{t}w_\delta(\cdot,T)-\varphi\partial_{t}w_\delta(\cdot,T)=(1-\varphi(\cdot))\partial_{t}w_\delta(\cdot,T).
\end{cases}
\end{equation*}
Taking into account that $\varphi=1$ in $\Omega_{\frac{\delta}{2}}^{\ast}$, we obtain
\begin{equation}\label{eq19.1.1}
\widetilde{u}(\cdot,T)=\partial_{t}\widetilde{u}(\cdot,T)=0 \hbox{ \ in \ } \Omega^{\ast}.
\end{equation}
Since we are interested in solving the control problem with initial data $(f_1,f_2)\in \widetilde{H}(\Omega^\ast)\times L^2(\Omega^\ast)$, it would be interesting if we had
\begin{equation*}
\widetilde{u}(\cdot,0)=f_1 \hbox{ \ and \ } \partial_{t} \widetilde{u}(\cdot,0)=f_2 \hbox{ \ in \ } \Omega^{\ast}.
\end{equation*}
This is equivalent to solving, for the unknown $(w_0,w_1) \in \widetilde{H}(\Omega^\ast)\times L^2(\Omega^\ast)$, the system
\begin{equation}
\begin{cases}\label{eq20.1.1}
E_1w_0-z(\cdot,0)=f_1 &~\hbox{ in }~ \Omega^\ast, \\
E_0w_1-\partial_{t}z(\cdot,0)=f_2 &~\hbox{ in }~ \Omega^\ast.
\end{cases}
\end{equation}
In terms of the operators $E, \ S_{T}, \ S_{T}^{\ast}$ and $M_\varphi$ we can rewrite this system as
\begin{equation}\label{eq21.1.1}
E(w_0,w_1)-S_{T}^{\ast}M_{\varphi}S_{T}E(w_0,w_1)=(f_1,f_2) \hbox{ \ in \ } \Omega^\ast.
\end{equation}
Let $R$ the restriction to $\Omega^\ast$, then the equation \eqref{eq21.1.1} can be written as
\begin{equation}\label{eq22.1.1}
(I-RS_{T}^{\ast}M_{\varphi}S_{T}E)(w_0,w_1)=(f_1,f_2).
\end{equation}
Introducing the operator $K_{T}=RS_{T}^{\ast}M_{\varphi}S_{T}E$, the equation \eqref{eq22.1.1} becomes
\begin{equation}\label{eq23.1.1}
(I-K_{T})(w_0,w_1)=(f_1,f_2).
\end{equation}
Next, we present a diagram with the definition of the operator $K_T$:
\begin{figure}[H]
\begin{center}
	\includegraphics{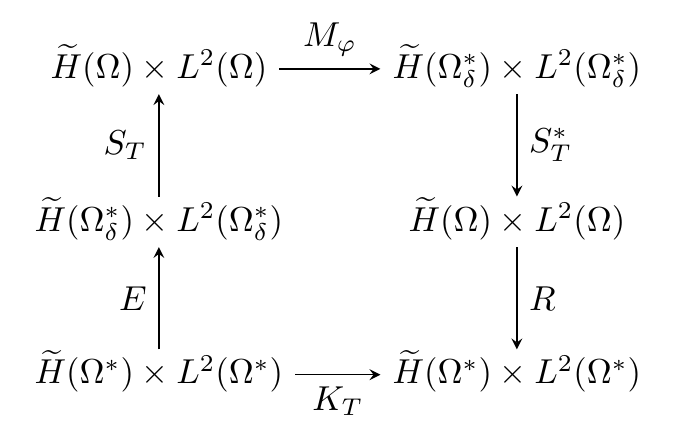}
\end{center}
	\caption{Definition of the operator $K_T$.}
\end{figure}

The operator $K_{T}:\widetilde{H}(\Omega^{\ast})\times L^2(\Omega^\ast)\rightarrow \widetilde{H}(\Omega^{\ast})\times L^2(\Omega^\ast)$ is linear and bounded. By Neumann's Theorem, the equation \eqref{eq23.1.1} has a solution if we prove that $K_{T}$ is a contraction for $T$ sufficiently large. Observe that,
\begin{align*}
\|K_T(w_0,w_1)\|_{\widetilde{H}(\Omega^{\ast})\times L^2(\Omega^\ast)}^{2}=&\|(z(\cdot,0),\partial_{t}z(\cdot,0)\|_{\widetilde{H}(\Omega^{\ast})\times L^2(\Omega^\ast)}^{2} \\
\leq & \|(z(\cdot,0),\partial_{t}z(\cdot,0)\|_{\widetilde{H}(\Omega_{\delta}^{\ast})\times L^2(\Omega_{\delta}^{\ast})}^{2} \\
=&\|S_{T}^{\ast}M_\varphi S_T E(w_0,w_1)\|_{\widetilde{H}(\Omega_{\delta}^{\ast})\times L^2(\Omega_{\delta}^{\ast})}^{2} .
\end{align*}
Using the estimate \eqref{eq12.1.1}, we obtain that
\begin{align*}
\|S_{T}^{\ast}M_\varphi S_T E(w_0,w_1)\|_{\widetilde{H}(\Omega_{\delta}^{\ast})\times L^2(\Omega_{\delta}^{\ast})}^{2} \leq&
\begin{cases}
C'e^{-2\gamma T}\|M_\varphi S_T E(w_0,w_1)\|_{\widetilde{H}(\Omega_{\delta}^{\ast})\times L^2(\Omega_{\delta}^{\ast})}^{2}, &~\hbox{ if }~ n ~\hbox{ is odd, }~  \\
C'T^{-2n}\|M_\varphi S_T E(w_0,w_1)\|_{\widetilde{H}(\Omega_{\delta}^{\ast})\times L^2(\Omega_{\delta}^{\ast})}^{2}, &~\hbox{ if }~ n ~\hbox{ is even, }~
\end{cases},
\end{align*}
Also,
\begin{equation*}
\|M_\varphi S_T E(w_0,w_1)\|_{\widetilde{H}(\Omega_{\delta}^{\ast})\times L^2(\Omega_{\delta}^{\ast})}^{2}\leq \alpha\|S_TE(w_0,w_1)\|_{\widetilde{H}(\Omega_{\delta}^{\ast})\times L^2(\Omega_{\delta}^{\ast})}^{2},
\end{equation*}
where $\alpha$ depends only $\varphi$.
Hence,
\begin{align*}
\|K_T(w_0,w_1)\|_{\widetilde{H}(\Omega^\ast)\times L^2(\Omega^\ast)}^{2} \leq&
\begin{cases}
\alpha C'e^{-2\gamma T}\| S_T E(w_0,w_1)\|_{\widetilde{H}(\Omega_{\delta}^{\ast})\times L^2(\Omega_{\delta}^{\ast})}^{2}, &~\hbox{ if }~ n ~\hbox{ is odd, }~  \\
\alpha C'T^{-2n}\| S_T E(w_0,w_1)\|_{\widetilde{H}(\Omega_{\delta}^{\ast})\times L^2(\Omega_{\delta}^{\ast})}^{2}, &~\hbox{ if }~ n ~\hbox{ is even, }~
\end{cases}.
\end{align*}
Using the estimate $\eqref{eq11.1.1}$, we have
\begin{equation*}
\| S_T E(w_0,w_1)\|_{\widetilde{H}(\Omega_{\delta}^{\ast})\times L^2(\Omega_{\delta}^{\ast})}^{2} \leq \begin{cases}
 C'e^{-2\gamma T}\| E(w_0,w_1)\|_{\widetilde{H}(\Omega_{\delta}^{\ast})\times L^2(\Omega_{\delta}^{\ast})}^{2}, &~\hbox{ if }~ n ~\hbox{ is odd, }~  \\
 C'T^{-2n}\| E(w_0,w_1)\|_{\widetilde{H}(\Omega_{\delta}^{\ast})\times L^2(\Omega_{\delta}^{\ast})}^{2}, &~\hbox{ if }~ n ~\hbox{ is even, }~
\end{cases}.
\end{equation*}
From the boundedness of the operator $E$, we obtain the existence of a positive constant $C''$, such that
\begin{align*}
 \|K_T(w_0,w_1)\|_{\widetilde{H}(\Omega^{\ast})\times L^2(\Omega^\ast)}^{2}
 \leq& \begin{cases}
C''e^{-4\gamma T}\|  (w_0,w_1)\|_{\widetilde{H}(\Omega^\ast)\times L^2(\Omega^\ast)}^{2}, &~\hbox{ if }~ n ~\hbox{ is odd, }~  \\
C''T^{-4n}\| (w_0,w_1)\|_{\widetilde{H}(\Omega^\ast)\times L^2(\Omega^\ast)}^{2}, &~\hbox{ if }~ n ~\hbox{ is even, }~
\end{cases},
\end{align*}
for $T>T_0$. Now we fix $T>T_0$ such that $C''e^{-4\gamma T}<1$ and $C''T^{-4n}<1$, so that $K_T$ is a contraction in $\widetilde{H}(\Omega^\ast)\times L^2(\Omega^\ast)$. Let $(w_0,w_1) \in \widetilde{H}(\Omega^\ast)\times L^2(\Omega^\ast)$ be the unique solution to \eqref{eq23.1.1}. Now we define
\begin{equation*}
(\widetilde{f}_1,\widetilde{f}_2)=E(w_0,w_1)-S_{T}^{\ast}M_{\varphi}S_{T}E(w_0,w_1)
\end{equation*}
and observe that $(\widetilde{f}_1,\widetilde{f}_2)$ is an extension of $(f_1,f_2)$ to $\Omega$.

Using these extensions as initial data, we solve the problem
\begin{equation}\label{eq24.1.1}
\begin{cases}
(\partial_{t}^{2}+ P)\widetilde{u}=0 &~\hbox{ in }~ \Omega \times \mathbb{R}, \\
\widetilde{u}=0 & ~\hbox{ on }~ \partial \Omega, \\
\widetilde{u}(x,0)=\widetilde{f}_1(x), ~\partial_t \widetilde{u}(x,0)=\widetilde{f}_2(x) &~\hbox{ in }~  \Omega,
\end{cases}
\end{equation}
and we note that $\widetilde{u}$ satisfies $\widetilde{u}(\cdot,T)=\partial_t\widetilde{u}(\cdot,T)=0$ in $\Omega^\ast$.

Observing that $\partial_{t}^2\widetilde{u}+P\widetilde{u} \in L_{loc}^{2}(\Omega^\ast \times ]0,T[)$ we conclude from Assumption \ref{Tataru} that the conormal derivative, $\partial_\nu \widetilde{u}$,  is square integrable over each smooth part of $\partial \Omega^\ast \times ]0,T[$. Pasting these traces we can define the desired control in $L^2(\partial \Omega^* \times ]0,T[)$. Now, defining $u:=\widetilde{u}|_{\Omega^\ast\times[0,T]}$ and $g:=B^\ast \widetilde{u} $ where $B^{\ast}\widetilde{u}=\alpha \widetilde{u}+\beta \partial_\nu \widetilde{u}$, for $\alpha, \beta \in \mathbb{R}$ and $\alpha^2+\beta^2\neq 0$, from the construction, we see that $u$ solves the problem
\begin{equation*}
\begin{cases}
(\partial_{t}^{2}+ P)u=0 &~\hbox{ in }~ \Omega^\ast \times (0,T), \\
B^\ast u = g &~\hbox{ on }~\partial \Omega^\ast \times (0,T),\\
u=0  &~\hbox{ on }~\partial \Omega \times (0,T),\\
u(x,0)=f_1(x), ~\partial_t u(x,0)=f_2(x) &~\hbox{ on }~ \Omega^\ast,
\end{cases}
\end{equation*}
and satisfies the conditions $u(\cdot,T)=\partial_{t}u(\cdot,T)=0$ in $\Omega^\ast$.
\end{proof}

\newpage

\section{Final Remarks}

The following section summarizes the new contributions of the present paper compared with the works cited in the introduction.

		\begin{table}[H]
			\begin{center}
				\resizebox{0.99999\textwidth}{!}{
		\begin{tabular}{|c | c | c | c| }
			
			\hline 		\multicolumn{4}{|c|}{ \fonte{Summary of the literature with respect to boundary controllability to problem $ u_{tt}-\Delta u=0 $}} \\ \hline
			\fonte{Authors}	& \fonte{ Control}     & \fonte{Setting} & \fonte{Tools/Comments}\\
			\hline
			\fonte{C. Bardos, G. Lebeau and J. Rauch \cite{9}} &   \makecell{ \fonte{$B$ is a differential operator} \\  \fonte{of degree zero or one with smooth coefficients,}\\ \fonte{and $\partial M$ is noncharacteristic for $B$.}}  & \fonte{ Riemannian.}
			& \begin{minipage} [c] {0.4\textwidth} \vspace{.2cm}
			\begin{itemize}
					\done \fonte{Smooth coeffients.}
					\done Microlocal Analysis.
					\done  Unique continuation.
					\done Ultra-weak solutions.
					\crossed Transmission Problem.
				\end{itemize}   \vspace{.2cm}
			\end{minipage}\\
			\hline
				\fonte	W. D. Bastos and A. Spezamiglio \cite{6} &  \fonte Robin  & \fonte Euclidean.
			&  \begin{minipage} [c] {0.4\textwidth} \vspace{.2cm}
				\begin{itemize}
					\done \fonte Curved polygon.
					\crossed Variable coefficients.
					\done Microlocal Analysis.
					\crossed Ultra-weak solutions.
					\crossed Transmission Problem.
				\end{itemize}  \vspace{.2cm}
			\end{minipage}\\
			\hline
			\fonte	N. Burq  \cite{18}&   \fonte Dirichlet & \fonte Riemannian.
			& \begin{minipage} [c] {0.4\textwidth} \vspace{.2cm}
				  \begin{itemize}
					\done \fonte  Microlocal analysis.
					\done Nonsmooth variable coefficients.
					\crossed Nonsmooth boundary.
					\crossed Transmission Problem.
				\end{itemize}   \vspace{.2cm}
			\end{minipage}\\
		\hline
	\fonte	F. Cardoso and G. Vodev \cite{23} &  \fonte $\times$  & \fonte Euclidean endowed with a Riemannian metric.
		&  \begin{minipage} [c] {0.4\textwidth} \vspace{.2cm}
			  \begin{itemize}
				\done  \fonte Local energy decay.
				\done Exponential decay.
				\done Resolvent estimates.
				\done Transmission Problem.
				\crossed Boundary exact controllability.
			\end{itemize}  \vspace{.2cm}
		\end{minipage}\\
			\hline
			\fonte	P. Grisvard \cite{8}& \fonte  Neumann & \fonte Euclidean.
			&  \begin{minipage} [c] {0.4\textwidth} \vspace{.2cm}
				  \begin{itemize}
					\crossed \fonte Microlocal analysis.
					\crossed Variable coefficients.
					\done Nonsmooth boundary.
					\crossed Transmission Problem.
					\done Mixed boundary conditions i. e., when singular solutions occur.
				\end{itemize}  \vspace{.2cm}
			\end{minipage}\\
			\hline
			\fonte	J. U. Kim \cite{40}& \fonte   Robin & \fonte Euclidean.
			&   \begin{minipage} [c] {0.4\textwidth} \vspace{.2cm}
				  \begin{itemize}
					\done \fonte  Microlocal analysis.
					\done Variable coefficients.
					\crossed Nonsmooth boundary.
					\crossed Transmission Problem.
					\done Trace regularity.
				\end{itemize}  \vspace{.2cm}
			\end{minipage}\\
			\hline
			\fonte J. Lagnese \cite{26}&  \fonte Dirichlet & \fonte Euclidean.
			&   \begin{minipage} [c] {0.4\textwidth} \vspace{.2cm}
				  \begin{itemize}
					\crossed  \fonte Microlocal analysis.
					\crossed Variable coefficients.
     				\done Piecewise constant coefficients.
					\crossed Nonsmooth boundary.
					\done Transmission Problem (two regions).
					\end{itemize}  \vspace{.2cm}
			\end{minipage}\\
			\hline
		\fonte	I. Lasiecka, R. Triggiani, and P. F. Yao \cite{21}, \cite{31}& \fonte  Dirichlet/Neumann & \fonte Riemannian.
			&   \begin{minipage} [c] {0.4\textwidth} \vspace{.2cm}
				  \begin{itemize}
					\done \fonte Carleman estimates.
					\done Interior controllability.
					\done Variable coefficients $(C^2/C^1)$.
					\crossed Nonsmooth boundary.
					\crossed Transmission Problem.
				\end{itemize}  \vspace{.2cm}
			\end{minipage}\\
			\hline
			\fonte	J. L. Lions \cite{10}&  \fonte Dirichlet & \fonte Euclidean.
			&   \begin{minipage} [c] {0.4\textwidth} \vspace{.2cm}
				  \begin{itemize}
					\done \fonte Interior controllability.
					\crossed Variable coefficients.
					\crossed Nonsmooth boundary.
					\done Transmission Problem (two regions).
				\end{itemize}  \vspace{.2cm}
			\end{minipage}\\
			\hline
		\fonte	S. Nicaise \cite{28}, \cite{29}	& \fonte Robin & \fonte Euclidean.
			&   \begin{minipage} [c] {0.4\textwidth} \vspace{.2cm}
				  \begin{itemize}
					\done \fonte Exact controllability of networks of elastic polygonal membranes.
					\crossed Variable coefficients.
					\crossed Microlocal Analysis.
									\end{itemize}  \vspace{.2cm}
			\end{minipage}\\
			\hline
			
		\fonte	D. L. Russell \cite{3}	& \fonte Dirichlet & \fonte Euclidean.
			&   \begin{minipage} [c] {0.4\textwidth} \vspace{.2cm}
				  \begin{itemize}
				\done \fonte Scattering theory results.
				\done Nonsmooth boundary.
				\crossed Neumann controllability.
				\end{itemize}  \vspace{.2cm}
			\end{minipage}\\
			\hline
		\fonte	{\bf Present article}	&  \fonte Robin &  \fonte \makecell{Euclidean domain with smooth obstacles,\\ \fonte endowed with a Riemannian metric.}
			&   \begin{minipage} [c] {0.4\textwidth} \vspace{.2cm}
				  \begin{itemize}
					\done \fonte Scattering theory results.
					\done Tataru's results about trace regularity.
					\done Variable jumped coefficients.
					\done Nonsmooth boundary.
					\done Transmission Problem for $n-$regions.
					\crossed Controllability for ultra-weak solutions.
				\end{itemize}  \vspace{.2cm}
			\end{minipage}\\
			\hline
		\end{tabular}}
	\end{center}
	\end{table}

\end{document}